\documentclass{amsart}

\usepackage{amsmath,amssymb,amsfonts,enumerate,amsthm,graphicx}

\usepackage{amscd}
\input xy
\xyoption{all}
\headsep=1cm \numberwithin{equation}{section}
 \newtheorem{thm}{Theorem}[section]
 \newtheorem{cor}[thm]{Corollary}
 \newtheorem{lem}[thm]{Lemma}
 \newtheorem{prop}[thm]{Proposition}
 \theoremstyle{definition}
 
 \newtheorem{example}[thm]{Example}
 \theoremstyle{remark}

 \numberwithin{equation}{section}
 
\DeclareMathOperator{\Hom}{Hom}

 \DeclareMathOperator{\Ass}{Ass}

\DeclareMathOperator{\miin}{Min}
\DeclareMathOperator{\Ext}{Ext} 
\DeclareMathOperator{\Tor}{Tor}
\DeclareMathOperator{\depth}{depth}
\DeclareMathOperator{\rank}{rank} \DeclareMathOperator{\Ht}{ht}
 
\DeclareMathOperator{\grade}{grade}
 
\DeclareMathOperator{\A}{\alpha}
\DeclareMathOperator{\G}{\gamma}\DeclareMathOperator{\Dt}{\delta}
\DeclareMathOperator{\E}{\textit{E}}

\DeclareMathOperator{\om}{\omega}
\DeclareMathOperator{\pd}{pd} \DeclareMathOperator{\ot}{\otimes}

\renewcommand{\P}{\mbox{P}}

\newcommand{\bt}{\beta}
\newcommand{\xra}{\xrightarrow}
\newcommand{\vf}{\varphi}

\newcommand{\fm}{\mathfrak{m}}

\newcommand{\fp}{\frak{p}}

\newcommand{\fa}{\frak{a}}

\newcommand{\ra}{\rightarrow}

\begin{document}
\bibliographystyle{amsplain}

\title[Ideals whose first two Betti numbers are close]
{Ideals whose first two Betti numbers are close}

\author{Keivan Borna}

\address{Keivan Borna: Faculty of Mathematical Sciences and Computer, Tarbiat Moallem University, Tehran,
Iran and School of Mathematics, Institute for Research in
Fundamental Sciences (IPM),P.O. 19395-5746, Tehran, Iran.}

\email{borna@ipm.ir}

\author{S. H. Hassanzadeh}

\address{Seyed Hamid Hassanzadeh:
Faculty of Mathematical Sciences and Computer, Tarbiat Moallem
University, Tehran, Iran and School of Mathematics, Institute for
Research in Fundamental Sciences (IPM),P.O. 19395-5746, Tehran,
Iran.}

\email{h\_hassanzadeh@tmu.ac.ir}

\thanks{The research of Borna was in part
supported by grant No. 88130035 from IPM}
\thanks{The research of Hassanzadeh was in part
supported by grant No. 88130112 from IPM}



\date{}

\begin{abstract}
For an ideal $I$ of a Noetherian local ring $(R,\fm,k)$ we show
that $\bt_1^R(I)-\bt_0^R(I)\geq -1$. It is demonstrated that some
residual intersections of an ideal $I$ for which
$\bt_1^R(I)-\bt_0^R(I)= -1\;\text{or}\;0$ are perfect. Some
relations between Betti numbers and Bass numbers of the canonical
module are studied.
\end{abstract}

\maketitle

\section*{Introduction}
The eventual behavior of the Betti sequence $\{\bt_i^R(M)\}$ has
been extensively studied, while the behavior of the initial Betti
numbers is not that well-known.

In the first section of this paper we are interested in studying
the relations between the first two Betti numbers. The motivations
for this section are two folds. The first, the inequality in
Proposition \ref{e:m1} which says that "for an ideal $I$ of a Noetherian local ring $(R,\fm,k)$,  $\bt_1^R(I)-\bt_0^R(I)\geq -1$" motivates us to study ideals where the
difference of their first two Betti numbers lives in the boarder.
The second comes from the roles of perfect ideals of height $2$
for which $\bt_1-\bt_0=-1$, and perfect Gorenstein ideals of
height $3$ where $\bt_1-\bt_0=0$, in the theory of {\it{residual
intersection}}.

As it is known, {\it{linkage theory}} is a basic tool to compare
and classify some families of algebraic structures. It involves
the direct calculations of the colon ideals $J=\fa:I$ where $\fa$
is a complete intersection ideal. Residual intersection arises
without some restriction and permits the comparison of objects in
very different dimensions. Precisely, we say that $J=\fa:I$ is an
$s$-residual intersection of $I$ if $\Ht J\geq s\geq\Ht I$
whenever $s\geq\mu(\fa)$. If in addition $\Ht(I+J)\geq s+1$, $J$
is a geometric $s$-residual intersection of $I$. Linkage preserves
the Cohen-Macaulay property (CM property). A famous question in
this area is ``When is a residual intersection ideal of $I$ CM?''
Answering this question involved a lot of attempts started by
Artin and Nagata \cite{AN}. Huneke \cite{H} showed that in a
Cohen-Macaulay local ring any geometric $s$-residual intersection
of an ideal $I$ which is {\it strongly Cohen-Macaulay} (or at least {\it sliding depth} \cite{HVV})and
satisfies the $G_{s+1}$ condition, is CM.
Trying to weaken the conditions imposed on $I$, Huneke and Ulrich
\cite{EHU} proved that if $I$ is evenly linked to an ideal which
satisfies $G_{s+1}$ and is strongly Cohen-Macaulay, then any
geometric $s$-residual intersection of $I$ is Cohen-Macaulay.
Examples of such ideals are perfect ideals of height $2$ and
perfect Goenstein ideals of height $3$. As it already pointed out,
perfect ideals of height $2$ has $\bt_1-\bt_0=-1$ and  perfect
Gorenstein ideals of height $3$ has $\bt_1-\bt_0=0$.

We show that in the presence of some $G_s$ condition, ideals with
small $\bt_1-\bt_0$ admit perfect geometric residual intersections
even if they do not satisfy the sliding depth condition. As well, we provide an example of
an (non CM) ideal with $\bt_1-\bt_0=0$ which admits a CM $3$-residual intersection but
does not satisfy the sliding depth condition. In the next propositions in the first section we try to identify some basic properties of ideals with $\bt_1-\bt_0=0,-1$.

Continuing to find more relations among initial Betti numbers, we
encountered with the Poincar\'{e} series of the canonical module
in the second section. In fact we deduced some relations between
Bass numbers of an ideal and Betti numbers of a canonical
module. Corollary \ref{cpoincare} which is one of the applications
of these relations shows that a partial sum of the Poincar\'{e} series of the canonical module $\omega_{R/J}$ is positive or negative if and only if the grade of $J$ is even or odd, respectively.


\section*{Setup}
Throughout this paper $(R,\fm,k)$ is a Noetherian local ring of
dimension $d$, $I$ is an ideal of  $R$ and  $M$ is a finitely
generated $R$-module. By $\bt_i^R(M)$ (or $\bt_i$ if it does not
arise any ambiguity) we mean the $i$th Betti number of $M$, i.e.,
$\dim_k\,\Tor_i^R(k,M)$. $P_M(t)=\sum_{i=0}^{\infty}\bt_i^R(M)t^i$
is the Poincar\'{e} series of $M$. For an integer $s$, the partial
sum of $P_M(t)$ is denoted by $P_M^{\leq
s}(t)=\sum_{i=0}^{s}\bt_i^R(M)t^i$. The $i$th Bass number of $M$,
i.e., $\dim_k\,\Ext^i_R(k,M)$ is denoted by $\mu^i_R(M)$ and
$\mu(M)$ is the minimal number of generators of $M$.
 The {\it {analytic spread}} of $I$,
denoted by $\ell(I)$, is the Krull dimension of
$\mathcal{R}(I)\otimes_R k$ , where $\mathcal{R}(I)$ is the Rees
algebra of $I$. One always has $\Ht(I)\leq\ell(I)\leq\min\{\dim
R,\mu(I)\}$, cf. \cite[1.89]{V}. Let $I=(i_1,\cdots,i_r)$ and
denote by $H_i(I)$ the $i$th homology of the Koszul complex with
respect to $i_1,\cdots,i_r$. We say that $I$ satisfies the {\it
sliding depth} condition (SD) if $\depth(H_i(I))\geq d-r+i, \;
i\geq0$; see \cite {HVV}, or \cite{HATez} for an alternate
definition. An ideal $I$
satisfies condition $G_{s+1}$ for some integer $s$, if $\mu(I_\fp)\leq\Ht \fp$ for all
prime ideal $I\subset \fp$ with $\Ht \fp\leq s$. Further, $I$
satisfies $G_{\infty}$ if $I$ satisfies $G_{s}$ for all $s$.
  $M$ is called {\it unmixed} if all
associated prime ideals of $M$ have a same height.  The {\it unmixed
part} of an ideal $I$, $I^{\text{unm}}$, is the intersection of
all primary components of $I$ with height equal to $\Ht(I)$. By a
perfect module we mean a CM module of finite projective dimension. Finally, we say that  $M$ has rank $r$
if $M\otimes_R Q$ is a free $Q$-module of rank $r$ where $Q$ is
the total ring of fractions of $R$.


\section{Small $\bt_1-\bt_0$}

The following proposition  provides one motivation for studying
modules with small $\bt_1-\bt_0$.

\begin{prop}\label{e:m1}
Let $I$ be an ideal of $R$. Then the followings hold:
\begin{enumerate}[\quad\rm(a)]
\item $\sum_{j=0}^{2n+1}(-1)^{j+1}\bt_{j}^R(I)\geq-1\;\text{for all
}n\geq0$; in particular $\bt_1^R(I)-\bt_0^R(I) \geq -1$.
\item If $R$ is unmixed and equality in (a) holds, then
$\grade(I)>0$.
\item If $\grade(I)>0$, then $\sum_{j=0}^{2n}(-1)^{j+1}\bt_{j}^R(I)\leq-1\;\text{for all
}n\geq0$.
\end{enumerate}
\end{prop}

\begin{proof}
\begin{enumerate}[\quad\rm(a)]
\item Let $\cdots \ra R^{\bt_1}\ra R^{\bt_0}\ra I\to 0$
be the minimal free resolution of $I$ and set
$Z_i=\ker(R^{\bt_i}\ra R^{\bt_{i-1}})$ for all $ i \geq 1$. By localizing the exact
sequence
$0\rightarrow Z_i\ra R^{\bt_i}\ra R^{\bt_{i-1}}\ra\cdots\ra I\to 0$
at the minimal prime $p$, we get the exact sequence
$$0\rightarrow (Z_i)_p\ra {R_p}^{\bt_i}\ra {R_p}^{\bt_{i-1}}\ra\cdots\ra IR_p\to 0$$
of $R_p$-modules.  $R_p$ is an Artinian local ring, hence each module in this exact sequence is of finite length.  We then have
$$length_{R_p}(IR_p)+\sum_{j=0}^i(-1)^{j+1}length_{R_p}({R_p}^{\bt_j})+(-1)^{i+2}length_{R_p}((Z_i)_p)=0.\qquad(\dag)$$

This equality in conjuction with  $length_{R_p}(R_p)\geq length_{R_p}(IR_p)$ yields
$$length_{R_p}(R_p)(1+\sum_{j=0}^i(-1)^{j+1}\bt_j)\geq-(-1)^{i+2}length_{R_p}((Z_i)_p).\qquad(\ddag)$$
Note that $length_{R_p}((Z_i)_p)\geq0$; so that  for  $i=2n+1$, an odd
integer, the desired inequality follows.

\item Assume that $R$ is unmixed and that $\sum_{j=0}^{2n+1}(-1)^{j+1}\bt_j=-1$ for some integer
$n\geq0$. Then it follows from ($\ddag$) that $(Z_{2n+1})_p=0$ for
all $p\in\Ass(R)$. Since $\Ass Z_i\subset\Ass R$ for all $i$,
$Z_{2n+1}=0$. Hence the projective dimension of $I$ is finite.
Thus $I$ has a rank and $\rank(I)=\sum_{j=0}^{2n+1}(-1)^j\bt_j=1$, therefore $\grade(I)>0$.

\item Since $\grade(I)>0$, there is an $x\in I$ which is
$R$-regular. Thus
$length_{R_p}(R_p)\geq length_{R_p}(IR_p)\geq length_{R_p}(xR_p)=length_{R_p}(R_p)$
for each $p\in\miin(R)$ (because $x\notin p$). This shows that
$length_{R_p}(IR_p)=length_{R_p}(R_p)$ for each $p\in\miin(R)$.
 employing this fact in ($\dag$), we get
$-(-1)^ilength_{R_p}((Z_i)_p)=length_{R_p}(R_p)(1+\sum_{j=0}^i(-1)^{j+1}\bt_j).$
 Hence for  $i=2n$ an even integer,
$\sum_{j=0}^i(-1)^{j+1}\bt_j\leq-1$.

\end{enumerate}
\end{proof}

The next theorem is our main theorem on the Cohen-Macaulayness of
residual intersections of ideals with small $\bt_1-\bt_0$, here we
do not care about the properties of the corresponding Koszul
complex.

\begin{thm}\label{e:m2}
Let $(R,\fm,k)$ be a Cohen-Macaulay Noetherian local ring, $I$ an
ideal of positive grade $g$ and with analytic spread $\ell$,
denote $\bt_i=\bt_i^R(I)$ for all $i$. Then
\begin{enumerate}[\quad\rm(a)]

\item if $\bt_1-\bt_0=-1$, then any $s$-residual intersection
of $I$ is perfect of projective dimension $s-1$;

\item if $\bt_1-\bt_0=-1$ and $R$ is Gorenstein, then  $I^{\text{unm}}$ is
perfect;

\item if $\bt_1-\bt_0=0$ and $I$ satisfies $G_{\ell+1}$ and $k$ is an infinite field,
 then $I$ admits a CM geometric $\ell$-residual intersection of  projective dimension $\ell$.

\end{enumerate}
\end{thm}

\begin{proof}
\begin{enumerate}[\quad\rm(a)]
\item Let $\fa\subset I$ be an ideal generated by $s$ elements
such that $J=\fa:I$ is an $s$-residual intersection of $I$.  $I$ has the presentation $0\to R^{\bt_0-1}\to
R^{\bt_0}\to I\to0$. Let $\cdots\to R^s\to\fa\to0$ be a free
resolution for $\fa$. Lifting the inclusion $\fa\subset I$ and
computing the mapping cone of this lifting map, we have an exact
complex
\[
\cdots\to R^{\bt_0+s-1}\xra{\psi}R^{\bt_0}\xra{}I/\fa\to0
\]
By Fitting theorem \cite[Lemma 20.7]{E01}, $J={\rm
Ann}(I/\fa)\subset\sqrt{I_{\bt_0}(\psi)}$. Thus $\grade
I_{\bt_0}(\psi)\geq\grade(J)\geq s=(\bt_0+s-1)-\bt_0+1$ the
largest possible value. Therefore \cite[Excercise 20.6]{E01}
implies that $I_{\bt_0}(\psi)=J$. Now the Eagon-Northcott complex
of $\psi$ provides a free resolution for $R/I_{\bt_0}(\psi)$ of
length $s$. To see the Cohen-Macaulayness of $J$ we recall a
theorem of Hochster and Eagon which states that if a determinantal
ideal attains its maximum height in a Cohen-Macaulay local ring
then it is CM; see \cite[A2.55]{E02}.

\item In the case that $R$ is Gorenstein, for a maximal regular
sequence contained in $I$ say $\A$, one has
$\A:J=\A:\A:I=\A:\A:I^{\text{unm}}=I^{\text{unm}}$, so that
$I^{\text{unm}}$ is linked to $J=\A:I$; see \cite[1.1.8
(ii)]{HATez}. Hence, by theorems in the Linkage theory, the fact
that $J$ is perfect implies that $I^{\text{unm}}$ is perfect.

\item This part is based on the following consequence of three
results of Eisenbud, Huneke and Ulrich \cite[1.1, 1.2, 3.7]{EHU}.
According to these results, if $I$ is an ideal of positive grade in a Noetherian
local ring with infinite residue field such that $I$ satisfies
$G_{\ell+1}$, then there exists an ideal $\fa\subset I$ generated
by $\ell$ elements which is a minimal reduction of $I$ such that
$\Ht(\fa:I)\geq\mu(\fa)+1=\ell+1$, note that since $\fa$ is a
reduction of $I$, $\Ht(I)=\Ht(\fa)\leq\mu(\fa)<\ell+1$, thus
$J=\fa : I$ is a $\ell$-residual intersection of $I$. The rest of
the proof is similar to that of part (a) just note that the
beginning  of the mentioned mapping cone is of the form
\[
R^{\bt_0+\mu(\fa)}\xra{\psi}R^{\bt_0}\to I/\fa\to0.
\]
Therefore $$\Ht I_{\bt_0}(\psi)\geq\Ht J\geq
\mu(\fa)+1=(\bt_0+\mu(\fa)-\bt_0+1)(\bt_0-\bt_0+1)$$ the greatest
possible value. The result will now follow as part (a).
\end{enumerate}
\end{proof}

The next example shows the benefit of Theorem \ref{e:m2}, to prove
the Cohen-Macaulayness of residual intersections in cases where
other approaches  can not be applied. For example the sliding
depth is a condition which often appears in the proofs of
Cohen-Macaulayness of residual intersections; see \cite{U} for
instance. In the following we give an example of an (non CM) ideal
which admits a CM $3$-residual intersection but does not satisfy
the sliding depth condition.

\begin{example} Let $R=\mathbb{Q}[x,y,z]$
and $I=(x^2,xy,z^2)$. Then $\Ht(I)=2$ and the only minimal prime
containing $I$ is $(x,z)$. One now can see that $I$ satisfies
$G_\infty$. The minimal free resolution of $I$ is
$$0\ra R(-5)\ra R(-3)\oplus R^2(-4)\ra R^3(-2)\ra0.$$ That is $\bt_1^R(I)-\bt_0^R(I)=0$, hence $I$
fulfills the conditions of Theorem \ref{e:m2}(c). Proposition 1.5
of \cite{V} describes the equations of the Rees algebra
$\mathcal{R}_I$, say $P$. Using the following procedure in CoCoA,
we are able to compute $P$.\\
 \textbf{Define} \textbf{Equations}(I)\\
 S::=Q[t[1..4],x[1..3],u];\\
 \textbf{Using} S \textbf{Do}\\
I := \textbf{Ideal}(\textbf{BringIn}(\textbf{Gens}(I)));\\
G := \textbf{Gens}(I);\\
P := \textbf{Elim}(u, \textbf{Ideal}([t[N]-u*G[N] | N \textbf{In} 1..\textbf{Len}(G)]));\\
  \textbf{Minimalize}(P);
  \textbf{Return} P;
 \textbf{EndUsing};
\textbf{EndDefine};\\
I:= \textbf{Ideal}(x[1]x[1],x[3]x[3],x[1]x[2]);\\
 $P$:=\textbf{Equations}(I);\\
  $P=(-t[3]x[1] +t[1]x[2], -t[2]x[1]^2 + t[1]x[3]^2, t[2]x[1]x[2]
- t[3]x[3]^2)$\\
We then obtain that $\ell(I)=3$. Hence Theorem \ref{e:m2}(c)
ensures that there is a geometric $3$-residual intersection of $I$
which is CM. On the other hand $I$ does not satisfy SD. To see
this notice that the first homology of the Koszul complex of the
above generating set of $I$ is $((x^2,z^2):I)/I=(x,z^2)/I$ which
has depth zero since $(x,y,z)=I:xz$.

\end{example}

 The next  lemma will be  helpful in the sequel.

\begin{lem}\label{l:rank}
Let $M$ be a finitely generated $R$-module which has a rank. Then
for all $i\geq0$,
\[\bt_i^R(M)-\bt_{i+1}^R(M)\leq\sum_{j=1}^{i}(-1)^{j-1}\bt_{i-j}^R(M)+(-1)^i
\rank(M).
\]
The equality holds if and only if $\bt_{i+2}^R(M)=0$.
\end{lem}

\begin{proof}
Let $i\geq0$ be an integer. Consider a minimal free resolution
$\mathbf{F}_{\bullet}$ for $M$ and let $Z_{i-1}$ to be the
$(i-1)$th syzygy module of this complex. (Take the ($-1$)th syzygy
of $M$ to be $M$ itself.) By \cite[Corollary 1.4.6]{BH}, $Z_{i-1}$
has a rank and
$$\rank(Z_{i-1})=\sum_{j=1}^{i}(-1)^{j-1}\bt_{i-j}^R(M)+(-1)^i\rank(M).$$
Now, the exact sequence $0\to Z_i\to R^{\bt_i}\to Z_{i-1}\to0$
implies that $Z_i$, as well, has a rank and
$\rank(Z_{i-1})+\rank(Z_{i})=\bt_i^R(M)$. Hence using the similar
fact that $\rank(Z_{i})=\bt_{i+1}^R(M)-\rank(Z_{i+1})$,  one gets
that
\begin{align*}
\begin{split}
\bt_i^R(M)-\bt_{i+1}^R(M)&=\rank(Z_{i-1})-\rank(Z_{i+1})\\&\leq\rank(Z_{i-1})=\sum_{j=1}^{i}(-1)^{j-1}\bt_{i-j}^R(M)+(-1)^i\rank(M).
\end{split}
\end{align*}
Clearly the equality holds if and only if $\rank(Z_{i+1})=0$ which
in turn implies that $Z_{i+1}=0$. i.e., $\bt_{i+2}^R(M)=0$.

\end{proof}
Employing the techniques  in the proof of  Theorem \ref{e:m2}, we
have the following properties of ideals with small $\bt_1-\bt_0$.

\begin{prop}\label{e:c1}
Let $I$ be an ideal of a Gorenstein local ring $R$ with
$\bt_1^R(I)-\bt_0^R(I)=-1$ and $\bt_1^R(I)\neq0$. Then $\pd (I)=1$
and moreover either $I$ is a perfect ideal of $\grade$ 2 or $I$ is
not unmixed and its grade is 1.
\end{prop}

\begin{proof}
Let $\bt_i=\bt_i^R(I)$. Since the equality in Lemma \ref{l:rank}
holds, we have $\bt_2=0$, that is $\pd\,I=1$. Now consider the
presentation $0\ra R^{{\bt_0}-1}\xra{\vf}{R^{\bt_0}}\ra I\ra0$. By
Hilbert-Burch theorem \cite[Theorem 1.4.7]{BH},
$I=aI_{{\bt_0}-1}(\vf)$ for some non-zero divisor $a$ of $R$, and
$I_{{\bt_0}-1}(\vf)$ is a perfect ideal of grade $2$. There are two
cases, either $a$ is unit or not. In the former case
$I=I_{{\bt_0}-1}(\vf)$ is a perfect ideal of grade $2$, whereas in
the latter one,  $\grade(I)=\grade(aI_{{\bt_0}-1}(\vf))
=\min\{\grade(Ra),\grade(I_{{\bt_0}-1}(\vf))\}=1$. Let $\A\in I$
be a non-zero divisor. Using the proof of Theorem \ref{e:m2}(a)
one can see that $(\A:I)$ is a determinant of a square matrix.
Thus it is a principle ideal say $(\G)$. By \cite[1.8
(ii)]{HATez}, $(\A:\G)=I^{\text{unm}}$. We note that $\A=\G\Dt$
for some non-zero divisor $\Dt$ of $R$, that is
$(\Dt)=I^{\text{unm}}$ which is an ideal of projective dimension
zero. In particular, $I\neq I^{\text{unm}}$.
\end{proof}

\begin{prop}
Let $I$ be an ideal of a Gorenstein local ring $R$ which satisfies
one of the following conditions:
\begin{enumerate}[\quad\rm(a)]
\item
$\bt_1^R(I)-\bt_0^R(I)=-1$ and $\grade(I)=1$ or
\item
$\bt_1^R(I)-\bt_0^R(I)=0$ and $\grade(I)=0$.
\end{enumerate}
Then $\om_{R/I}\cong R/I^{\text{unm}}$.
\end{prop}

\begin{proof}
For (a), by the same token as Corollary \ref{e:c1}, for a regular
element $\A$ in $I$, we have
$$\om_{R/I}=\frac{\A:I}{(\A)}=\frac{(\G)}{(\A)}\cong\frac{R}{0:_R(\G/\A)}=\frac{R}{(\A):(\G)}=\frac{R}{\A:(\A:I)}=\frac{R}{I^{\text{unm}}}.$$
For (b) according to the proof of Theorem \ref{e:m2}(a), one
obtains that $(0:I)$ is principle say $(\G)$. Hence
$\om_{R/I}=0:I=(\G)=R/(0:\G)=R/I^{\text{unm}}$.
\end{proof}


\section{Betti numbers of the canonical module}
This section is mostly devoted to study the behavior of the Betti
sequence of the canonical module of a Cohen-Macaulay ring $R/J$ ,
where $R$ is a Gorenstein ring,  specially those in which
$\grade(J)$ is small.

\begin{thm}\label{e:m3}
Let $(R,\fm,k)$ be a CM local ring of dimension $d$ and $J$ an
ideal of $R$ of height g. Let $\A$ be a maximal regular sequence
in $J$ and $I=\A:J$. Then

\begin{enumerate}[\quad\rm(a)]

\item $\bt_i^R(\om_{R/J})=\bt_{i+1}^R(R/I)$ for all $i\geq g+1$.

\item $\P_{\om_{R/J}}^{\leq g}(-1)=(-1)^g\bt_{g+1}^R(R/I)-\P_{R/I}^{\leq
g}(-1)$.

\item If $R$ is Gorenstein and $J$ is Cohen-Mcaulay, then $\bt_i^R(\om_{R/J})=\mu^R_{d-g+i}(R/J)$ for all $i\geq0$.

\item $\bt_0^R(\om_{R/J})=\mu_{d-g}^R(R/J)\mu_R^d(R)$.

\end{enumerate}
\end{thm}

\begin{proof}
\begin{enumerate}[\quad\rm(a)]

\item Considering the assumptions, we have
$I/\A=(\A:J)/\A=\Hom_R(R/J,R/\A)=\Ext^g_R(R/J,R)\cong\om_{R/J}$.
Hence we obtain the following exact sequence:
$$0\ra\om_{R/J}\ra R/\A\ra R/I\ra0$$
The long exact sequence of $\Tor$ yields that
$\Tor_{g+i+1}^R(R/\fm,R/I)\cong\Tor_{g+i}^R(R/\fm,\om_{R/J})$ for
all $i\geq1$ which finishes the proof of (a).

\item
 We consider the beginning terms of the long $\Tor$
exact sequence:
$$0\ra\Tor_{g+1}^R(k,R/I)\ra\Tor_{g}^R(k,\om_{R/J})\ra\Tor_{g}^R(k,R/\A)\ra\cdots\ra\Tor_0^R(k,R/I)\ra0.$$
Counting the vector space dimensions, we have
$$\bt_{g+1}^R(R/I)-\sum_{i=0}^{g}(-1)^i\bt_{g-i}^R(\om_{R/J})+\sum_{i=0}^{g}(-1)^i\bt_{g-i}^R(R/\A)-\sum_{i=0}^{g}(-1)^i\bt_{g-i}^R(R/I)=0.$$
The fact that
$\sum_{i=0}^{g}(-1)^i\bt_{g-i}^R(R/\A)=\sum_{i=0}^{g}(-1)^i\binom{g}{g-i}=0$
implies that $$\bt_{g+1}^R(R/I)-(-1)^g\P_{\om_{R/J}}^{\leq
g}(-1)-(-1)^g\P_{R/I}^{\leq g}(-1)=0$$ which is the assertion of
(b).

\item We construct the following spectral sequence of
Foxby. Consider the finite injective resolution of $R$,
$\mathbf{E}^\bullet:0\ra\E^0\ra\cdots\ra\E^d\ra0$ and the
projective resolution of $k$,
$\mathbf{P}_\bullet:\cdots\ra\P_1\ra\P_0\ra0$. The double complex
$\P_\bullet\ot_R\Hom_R(R/J,\E^\cdot)$ in the second quadrant is
pictured in the following:

$$\xymatrix{
        &\ar[d]\vdots                               & \ar[d]\vdots                               &\ar[d]\vdots                            &\\
0\ar[r] &P_1\ot_R\Hom_R(R/J,\E^0) \ar[d]\ar[r]\cdots&  P_1\ot_R\Hom_R(R/J,\E^{d-1}) \ar[d]\ar[r] & P_1\ot_R\Hom_R(R/J,\E^d)\ar[d] \ar[r] & {0} \\
0\ar[r] &P_0\ot_R\Hom_R(R/J,\E^{0})\ar[d]\ar[r]\cdots
&P_0\ot_R\Hom_R(R/J,\E^{d-1})\ar[d] \ar[r] &
P_0\ot_R\Hom_R(R/J,\E^d)\ar[d] \ar[r] & 0 \\
       &0                                           &0                                            &0& }$$
 The terms of the second horizontal spectral sequence arisen from
this double complex is as the following diagram,
$^{2}\E_{\text{hor}}^{-i,j}=\Tor_j^R(k,\Ext^{d-i}_R(R/J,R))$:

$$\xymatrix{
\vdots&\cdots&\vdots&\vdots\\
\Tor_2^R(k,\Ext^g_R(R/J,R)) & \cdots  & \Tor_2^R(k,\Ext^{d-1}_R(R/J,R)) & \Tor_2^R(k,\Ext^d_R(R/J,R))\ar[ddl] \\
\Tor_1^R(k,\Ext^g_R(R/J,R)) & \cdots  & \Tor_1^R(k,\Ext^{d-1}_R(R/J,R)) & \Tor_1^R(k,\Ext^d_R(R/J,R)) \\
\Tor_0^R(k,\Ext^g_R(R/J,R)) & \cdots
 &\Tor_0^R(k,\Ext^{d-1}_R(R/J,R))& \Tor_0^R(k,\Ext^d_R(R/J,R)) }$$

The fact that $R$ is Gorenstein and $R/J$ is Cohen-Macaulay of
dimension $d-g$ implies that $\Ext^{d-i}_R(R/J,R)=0$ for all
$i\neq d-g$. Thus this spectral sequence has infinite terms
$^{\infty}\E_{\text{hor}}^{-i,j}=0$ for $i\neq d-g$ and
$^{\infty}\E_{\text{hor}}^{-(d-g),j}=^{2}\E_{\text{hor}}^{-(d-g),j}$.

On the other hand the functorial isomorphism
$M\ot_R\Hom_R(N,I)\cong\Hom_R(\Hom_R(M,N),I)$, for a finitely
generated $R$-module $M$ and an injective $R$-module $I$ implies
the second vertical spectral sequence to be of the following form:

$$\xymatrix{
\vdots&\vdots&\vdots\\
\Ext^{d-2}_R(\Ext^1_R(k,R/J),R)&\Ext^{d-1}_R(\Ext^1_R(k,R/J),R) & \Ext^d_R(\Ext^1_R(k,R/J),R) \\
\Ext^{d-2}_R(\Ext^0_R(k,R/J),R)\ar[rru]&\Ext^{d-1}_R(\Ext^0_R(k,R/J),R)&\Ext^d_R(\Ext^0_R(k,R/J),R)
}$$

We now notice that $\dim\Ext^i_R(k,R/J)=0$ for all $i$, in
particular $\depth\Ext^i_R(k,R/J)=0$ for all $i$ and that
$\Ext^j_R(\Ext^i_R(k,R/J),R)=0$ for all $j\neq d$. Thus
$$
^{\infty}\E_{\text{ver}}^{-i,j}=
\left\{%
\begin{array}{ll}
    0, & \hbox{$i\neq0$;} \\
    \Ext^d_R(\Ext^j_R(k,R/J),R), & \hbox{$i=0$.} \\
\end{array}%
\right.
$$
Now the convergence of these two spectral sequences imply that
$$\Tor_i^R(k,\om_{R/J})\cong\Ext^d_R(\Ext^{d-g+i}_R(k,R/J),R)\;\text{for
all}\; i\geq0.$$
Accordingly,
$\bt_i^R(\om_{R/J})=\mu_R^{d-g+i}(R/J)\mu_R^d(R)=\mu_R^{d-g+i}(R/J)$
for all $i\geq0$.

\item To see this part notice that in $^{2}\E_{\text{hor}}$,
$^{2}\E_{\text{hor}}^{-g,0}$ is located in the right-down non-zero
corner of the diagram of $^{2}\E_{\text{hor}}$ (recall that
$\Ext^{d-i}_R(R/J,R)=0$ for all $i\neq d-g$); hence
$^{2}\E_{\text{hor}}^{-g,0}=$ $^{\infty}\E_{\text{hor}}^{-g,0}$.
Then the convergence of the spectral sequence implies that
$$\Tor_0^R(k,\Ext^{d-g}_R(R/J,R))\cong\Tor_0^R(k,\om_{R/J})\cong\Ext^d_R(\Ext^{d-g}_R(k,R/J),R)$$
which implies the assertion.
\end{enumerate}
\end{proof}

As a corollary of Theorem \ref{e:m3} in conjunction with Lemma
\ref{l:rank}, we obtain the following property of the Poincar\'{e}
series of the canonical module.

\begin{cor}\label{cpoincare}
Suppose that $(R,\fm)$ is a Cohen-Macaulay local ring and $J$ is an
ideal of $\grade$ $g\geq1$. Then
\begin{enumerate}[\quad\rm(a)]
\item if $g$ is odd, $\P_{\om_{R/J}}^{\leq g}(-1)\leq0$. In
particular, if $g=1$, $\bt_1^R(\om_{R/J})\geq\bt_0^R(\om_{R/J})$;

\item if $g$ is even, $\P_{\om_{R/J}}^{\leq g}(-1)\geq0$.

\end{enumerate}
\end{cor}

\begin{proof}
For (a) by Lemma \ref{l:rank},
$\bt_g^R(R/I)-\bt_{g+1}^R(R/I)\leq\bt_0^R(R/I)-\bt_1^R(R/I)+\cdots+\bt_{g-1}^R(R/I)$.
On the other hand by Theorem \ref{e:m3}(b), $\P_{\om_{R/J}}^{\leq
g}(-1)=\bt_g^R(R/I)-\bt_{g+1}^R(R/I)-(\bt_0^R(R/I)-\bt_1^R(R/I)+\cdots+\bt_{g-1}^R(R/I))$
which yields the assertion. For (b), $\P_{\om_{R/J}}^{\leq
g}(-1)=\bt_{g+1}^R(R/I)+(-\bt_0^R(R/I)+\cdots+\bt_{g-1}^R(R/I))-\bt_g^R(R/I)$
which is non-negative by Lemma \ref{l:rank}.
\end{proof}

\textbf{Acknowledgements}.  The authors would like to thank
Professor M-T. Dibaei for his useful comments which brought
improvement in the Proposition \ref{e:m1}.


\end{document}